\documentclass[12pt]{amsart}

\usepackage{amssymb,bm,mathrsfs,mathtools,booktabs}
\usepackage{enumerate}
\usepackage{enumitem}
\usepackage{tikz}
\usepackage{graphicx}
\usepackage{verbatim}
\usepackage{csquotes}
\usepackage{dsfont}
\usepackage{xcolor}

\usepackage[centering,width=6in]{geometry}
\usepackage{thmtools}
\usepackage{thm-restate}

\usepackage[hypertexnames=false]{hyperref}
\hypersetup{
    colorlinks=true,
    citecolor=blue,
    filecolor=black,
    linkcolor=blue,
    urlcolor=black
}

\numberwithin{equation}{section}

\newcommand{\R}{{\mathbb R}}
\newcommand{\Z}{{\mathbb Z}}
\newcommand{\Q}{{\mathbb Q}}

\newcommand{\T}{{\mathbb T}}

\newcommand{\rank}{{\rm rank}}

\DeclareMathOperator{\diam}{diam}
\DeclareMathOperator{\supp}{supp}

\newtheorem{theorem}{Theorem}[section]
\newtheorem{lemma}[theorem]{Lemma}

\newtheorem{corollary}[theorem]{Corollary}

\newtheorem{defn}{Definition}[section]

\theoremstyle{remark}

\newtheorem{question}{Question}[section]

\usepackage[
  backend=biber,
  style=numeric-comp,
  giveninits=true,
  uniquename=false,
  sorting=nyt,
  maxnames=99,
  doi=false,
  url=false,
  isbn=false,
  eprint=false
]{biblatex}

\DeclareDelimFormat{multinamedelim}{\addcomma\space}
\DeclareDelimFormat{finalnamedelim}{\addspace\bibstring{and}\space}

\DeclareFieldFormat{title}{#1}
\DeclareFieldFormat[article,inbook,incollection,inproceedings,book,misc,unpublished]{title}{#1}
\DeclareFieldFormat{journaltitle}{#1}
\DeclareFieldFormat{booktitle}{#1}

\DeclareFieldFormat{pages}{#1}
\DeclareFieldFormat{pagetotal}{#1\addspace pp}

\renewbibmacro{in:}{}

\AtEveryBibitem{%
  \clearfield{doi}%
  \clearfield{url}%
  \clearfield{isbn}%
  \clearfield{issn}%
  \clearfield{urldate}%
  \clearfield{mrnumber}%
  \clearfield{mrclass}%
  \clearfield{mrreviewer}%
  \clearfield{fjournal}%
}

\DeclareBibliographyDriver{article}{%
  \usebibmacro{bibindex}%
  \usebibmacro{begentry}%
  \printnames{author}%
  \setunit{\addcomma\space}%
  \printfield{title}%
  \setunit{\addcomma\space}%
  \printfield{journaltitle}%
  \setunit{\addspace}%
  \printfield{volume}%
  \iffieldundef{number}
    {}
    {\setunit{\addspace}\printtext{\mkbibparens{\printfield{number}}}}%
  \iffieldundef{year}
    {}
    {\setunit{\addspace}\printtext{\mkbibparens{\printdate}}}%
  \iffieldundef{pages}
    {}
    {\setunit{\addspace}\printfield{pages}}%
  \usebibmacro{finentry}%
}

\DeclareBibliographyDriver{book}{%
  \usebibmacro{bibindex}%
  \usebibmacro{begentry}%
  \ifnameundef{author}
    {\printnames{editor}}
    {\printnames{author}}%
  \setunit{\addcomma\space}%
  \printfield{title}%
  \setunit{\addcomma\space}%
  \printfield{series}%
  \setunit{\addcomma\space}%
  \printlist{publisher}%
  \setunit{\addcomma\space}%
  \printlist{location}%
  \setunit{\addcomma\space}%
  \printfield{year}%
  \iffieldundef{pagetotal}
    {}
    {\setunit{\addcomma\space}\printfield{pagetotal}}%
  \usebibmacro{finentry}%
}

\title[Common tiling functions with small support]{Common tiling functions with small support}

\author{Kailing Lai}
\address[Kailing Lai]{School of Mathematics and Information Science, Guangzhou University, Guangzhou, 510006, P.~R.~China,\newline and \newline
\href{http://math.uoc.gr/en/index.html}{Department of Mathematics and Applied Mathematics}, University of Crete,\\Voutes Campus, 70013 Heraklion, Greece}
\email{lkailinggl@163.com}

\begin{document}
\begin{abstract}
For $N$ lattices in $\R^d$ with volume $1$ and pairwise trivial intersections, every nonzero common tiling function has support diameter $\Omega(N^{1/d})$, while for lattice families whose fundamental domains have uniformly bounded diameters, the standard convolution construction gives an $O(N)$ upper bound, leaving a gap that has remained open since the work of Kolountzakis and Wolff \cite{kolwolff-1999Mathematika}. We close this gap by constructing, for every $d\geq 2$ and all sufficiently large $N$, lattice families satisfying the same volume and intersection conditions that admit a nonnegative common tiling function with support diameter $O(N^{1/d})$, thereby also answering Question 1 of Kolountzakis and Papageorgiou \cite{kolPapageorgiou-functions-2022jfaa}.
We also obtain the optimal $O(\sqrt N)$ upper bound by constructing, for any prescribed family of plane lattices whose volumes lie in a fixed bounded set independent of $N$, a pairwise trivially intersecting family with the same respective volumes and with bases arbitrarily close to suitable bases of the prescribed lattices.
\end{abstract}

\date{\today}
\maketitle

\tableofcontents

\section{Introduction}
\subsection{The Steinhaus tiling problem}
The classical Steinhaus problem for $\Z^2$ \cite{Moser,sierpi-sur-un-problem-1959fund}  asks for a set $E\subseteq\R^2$ having exactly one point in each translate of $R_\theta\Z^2$, for every rotation angle $\theta$. Equivalently, for each fixed rotation, every translate of the corresponding rotated lattice contains exactly one point of $E$.

The problem has two standard formulations: a set-theoretic formulation and a measurable formulation.
In the set-theoretic
formulation, no measurability assumption is imposed and the tiling is
required to be exact: for every rotation $R_\theta$, the translates of
$E$ by $R_\theta\Z^2$ form a partition of $\R^2$. Jackson and Mauldin proved the existence of a set-theoretic Steinhaus set in the plane \cite{jacksonsteve-2002amer, jacksonsteve-2002proc, jacksonsteve-2003bull}.  
%Ciucu \cite{Ciucu1996} and Srivastava and Thangadurai \cite{SrivastavaThangadurai2005} obtained topological restrictions on exact Steinhaus sets.
In the measurable formulation, $E$ is required
to be Lebesgue measurable and, for each $\theta$, the translates of $E$ by $R_\theta\Z^2$ form a partition of $\R^2$ up to a null set.
The plane measurable problem remains open. Several results are known. Croft \cite{croft1982quart} and Beck \cite{beckjozsef1989studia} ruled out bounded measurable Steinhaus sets in the plane. Subsequent work of Kolountzakis \cite{kol-anewestimate-1996internat,kol-aproblemof1996} and of Kolountzakis and Wolff \cite{kolwolff-1999Mathematika} imposed quantitative constraints on how rapidly a measurable planar solution could decay at infinity. In higher dimensions, Kolountzakis and Wolff \cite{kolwolff-1999Mathematika} showed that measurable Steinhaus sets do not exist for $d\geq 3$. A different proof of this higher-dimensional nonexistence result was later obtained by Kolountzakis and Papadimitrakis \cite{kol-thesteinhaustilingproblem-2002illinois}.

%Mauldin and Yingst \cite{MauldinYingst2003} studied measurable Steinhaus problems for more general lattices, and Chan and Mauldin \cite{ChanMauldin2007} proved nonexistence results for lattices equivalent to integral lattices in dimensions $d\geq 3$.

\subsection{Lattices and tiling}
A lattice $\Lambda\subseteq \R^d$ is a discrete additive subgroup of $\R^d$. It admits a representation
$\Lambda=B\Z^r,$
where $B\in\R^{d\times r}$ has rank $r$. The integer $r$, called the rank of $\Lambda$, is the dimension of the smallest real subspace of $\R^d$ containing $\Lambda$. The choice of $B$, whose columns generate the lattice $\Lambda$, is not unique: for $U\in\mathrm{GL}(r,\Z)$, the matrix $BU$ gives another basis matrix  for the same lattice. A lattice is said to be full rank when $r=d$, in which case $B\in\mathrm{GL}(d,\R)$. For such a lattice, the value $|\det B|$ is independent of the choice of basis matrix and is called the volume of $\Lambda$, denoted by $\text{vol}(\Lambda)$.
For a full-rank lattice $\Lambda=B\Z^d$, the dual lattice is
%\begin{equation}\label{dual lattice}
\[\Lambda^\ast=
\left\{\xi\in\R^d:\langle \xi,\lambda\rangle\in\Z
\text{ for all }\lambda\in\Lambda
\right\}
=B^{-T}\Z^d.\]
%\end{equation}
Throughout this paper, all lattices are assumed to be of full rank.

Let $\Lambda$ be a lattice in $\R^d$. Writing
$\Lambda=B\Z^d$, where $B=(b_1,\ldots,b_d)$ with $b_i\in\R^d$, the basis matrix $B$ determines the fundamental parallelotope
\[
P_B=B[0,1)^d=\left\{\sum_{i=1}^{d}t_i b_i:
0\leq t_i<1,\ i=1,\ldots,d\right\}.
\]
The set $P_B$ is a measurable fundamental domain for $\Lambda$: each $x\in\R^d$ can be written uniquely as
$x=p+\lambda,$
with $p\in P_B$ and $\lambda\in\Lambda$, and
$|P_B|=|\det B|.$
Consequently, the translates $P_B+\lambda$, $\lambda\in\Lambda$, partition $\R^d$. This decomposition motivates the corresponding notion for general measurable sets.

A measurable set $E\subseteq\R^d$ tiles with $\Lambda$ at a constant level if
\[
\sum_{\lambda\in\Lambda}\mathbf 1_E(x-\lambda)=\text{const.},
\quad \text{for almost every }x\in\R^d.
\]
In this case, the family
$E+\Lambda:=\{E+\lambda:\lambda\in\Lambda\}$
is called a tiling of $\R^d$ at a constant level. 
When the constant is $1$, the translates
$\{E+\lambda:\lambda\in\Lambda\}$
cover $\R^d$ up to a null set, while any two distinct translates overlap only on a set of measure zero. In this case, $E$ is a measurable fundamental domain for $\Lambda$, up to null sets.

More generally, a function  $f\in L^1(\R^d)$ is said to tile with $\Lambda$ at a constant level if
\[
\sum_{\lambda\in\Lambda} f(x-\lambda)=\text{const.},\quad\text {for almost every } x\in\R^d.
\]
We usually say that $f+\Lambda$ is a tiling of $\R^d$. When $f=\mathbf 1_E$ for a measurable set $E$, this reduces to the corresponding translational tiling problem for measurable sets.
Throughout the paper, the tiling level is
assumed to be nonzero.

\subsection{Common fundamental domains for finitely many lattices}
For each fixed rotation $R_\theta$, the Steinhaus condition requires $E$ to be a fundamental domain for $R_\theta\Z^2$. Replacing the full family of rotated lattices by a finite collection leads to the problem of finding a common fundamental domain for several lattices.
Let $
\Lambda_1,\ldots,\Lambda_N\subseteq\R^d$
be full-rank lattices. If a measurable set $E$ is a fundamental domain for each
$\Lambda_i$, then its Lebesgue measure satisfies
$|E|=\text{vol}(\Lambda_i),
\ i=1,\ldots,N.$
Thus equality of the lattice volumes is a necessary condition.

Kolountzakis
\cite{kol-multilattice-1997internat}
proved that lattices $\Lambda_1,\ldots,\Lambda_N$ of equal volume admit a measurable common fundamental domain provided that
$\Lambda_1^\ast+\cdots+\Lambda_N^\ast$
is direct. For two lattices, Han and Wang
\cite{hanwang-2001geom}
removed the direct-sum hypothesis and proved that equality of volumes alone is sufficient for the existence of a measurable common fundamental domain. The common fundamental domains constructed by Kolountzakis
\cite{kol-multilattice-1997internat}
and by Han and Wang
\cite{hanwang-2001geom}
are generally unbounded. The existence of a bounded measurable common fundamental domain for an arbitrary pair of full-rank lattices with the same volume therefore remained unresolved. More recently, Grepstad and Kolountzakis
\cite{grepstad-kolountzakis-2026}
showed that boundedness can always be achieved for two full-rank lattices in $\R^d$ having the same volume. The restriction to two lattices is essential. Kolountzakis
\cite{kol-multilattice-1997internat}
constructed three lattices of the same volume in $\R^2$ that admit no common measurable fundamental domain.

\subsection{Common tiling functions for finitely many lattices}
Common fundamental domains constitute a special case of common tiling
functions. The connection is already visible at the level of indicator functions: if $E$ is a measurable fundamental domain for
each of the lattices $\Lambda_1,\ldots,\Lambda_N$, then
$\mathbf 1_E$ tiles with each $\Lambda_i$.

Kolountzakis and Wolff
\cite{kolwolff-1999Mathematika}
showed that a function $f\in L^1(\R^d)$ tiles with $\Lambda$ if and only if $\widehat f$ vanishes at every nonzero point of $\Lambda^*$. Applying this criterion to each lattice, $f$ tiles with all of $\Lambda_1,\ldots,\Lambda_N$ if and only if 
\begin{equation}\label{tile-all-lattices}
  \hat{f}=0\quad \text{on}\ \bigcup_{i=1}^{N}\Lambda_i^*\setminus\{0\}.  
\end{equation}

The Fourier characterization immediately yields a general convolution construction. Let
$f_i\in L^1(\R^d)$ tile with $\Lambda_i$, and set
\[
f=f_1\ast\cdots\ast f_N.
\]
Since
$\widehat f=\prod_{i=1}^{N} 
\widehat f_i$ and each $\widehat f_i$ vanishes at every nonzero point of $\Lambda_i^*$, the function $\widehat f$ vanishes at every nonzero point of each dual lattice $\Lambda_i^*$. By the preceding criterion, $f$ tiles with all of $\Lambda_1,\ldots,\Lambda_N$.
In particular, let $D_i$ be a bounded measurable
fundamental domain for $\Lambda_i$ and take
$f_i=\mathbf 1_{D_i}$. Then
\begin{equation}\label{convolution}
f
=
\mathbf 1_{D_1}*\cdots *\mathbf 1_{D_N}
\end{equation}
belongs to $L^1(\R^d)$ and is a nonzero common tiling function for
$\Lambda_1,\ldots,\Lambda_N$. Thus, even when a finite family of
lattices admits no measurable common fundamental domain, it still admits a common
tiling function. 
Moreover, if the fundamental domains $D_i$ can be chosen so that their diameters are uniformly bounded by a constant $C$ independent of $N$, then the convolution function defined in \eqref{convolution} satisfies
\[
\diam\supp f\leq C\cdot N.
\]
On the other hand, Kolountzakis and Wolff
\cite{kolwolff-1999Mathematika}
proved that there exists a positive constant $C_d$ such that, if
$\Lambda_1,\ldots,\Lambda_N\subseteq\ R^d$ have volume one and satisfy
\begin{equation}\label{disjoint-of-Lambda}
\Lambda_i\cap\Lambda_j=\{0\},
\quad \text{for all } i\neq j,
\end{equation}
then every nonzero function $f\in L^1(\R^d)$ that tiles with all the lattices satisfies
\begin{equation}\label{lower bound}
\diam\supp f\geq C_d N^{1/d}.
\end{equation}
Thus, for lattice families of volume one satisfying the pairwise trivial-intersection condition, the lower bound is of order $N^{1/d}$. However, it is also shown in \cite{kolwolff-1999Mathematika} that, no matter how the
fundamental domains $D_i$ are chosen, the construction
\eqref{convolution} gives
\[
\diam\supp f \ge c_d\cdot N,
\]
where $c_d$ depends only on the dimension $d$.
Motivated by the gap between these two bounds, Kolountzakis and Papageorgiou \cite{kolPapageorgiou-functions-2022jfaa} leave the following question open:

\begin{question}\label{question}
Can the gap between the lower bound \eqref{lower bound} and the linear upper bound $O(N)$ achievable by the convolution tile \eqref{convolution} be bridged? Are there examples of $\Lambda_i$, $i=1,2,\ldots,N$, satisfying \eqref{disjoint-of-Lambda} and a nonzero function $f\in L^1(\R^d)$ that tiles with all $\Lambda_i$ and such that
\[\diam \supp f=o(N)?\]
\end{question}
We give an affirmative answer to this question. More precisely, for every fixed $d\ge 2$, we construct lattice families satisfying the pairwise intersection condition and a common tile whose support diameter is of order $N^{1/d}$.

\subsection{Main results}
We now state our main results. When $d=1$, the only lattice of volume 1 is $\Z$. In this case, $f=\mathbf 1_{[0,1)}$ tiles with $\Z$, and $\diam\supp f=1$. Therefore we will only consider results for $d\ge 2$. 

\begin{theorem}\label{1/d}
For $d\ge 2$ and each $N\ge 2$, there exist $N$ lattices $\Lambda_1,\ldots ,\Lambda_N\subseteq \R^d$ of volume $1$ satisfying
\[\Lambda_i\cap\Lambda_j=\{0\}, \quad  \text{for  all } \ i\ne j,\]
and a nonnegative function $f\in L^{1}(\R^d)$ that tiles with all $\Lambda_i$ and satisfies
\begin{equation}%\label{optimal}
\diam \supp f= O(N^{1/d}).  \nonumber 
\end{equation}
\end{theorem}
The proof of this theorem proceeds as follows. We start from a family of plane lattices that intersect only at the origin. By means of auxiliary matrices, we then construct higher‑dimensional lattices of volume 1 with the required properties and the required common tiling function. The two‑dimensional construction is particularly crucial for completing the proof. We shall therefore next consider a result concerning families of plane lattices.

\begin{theorem}\label{-close}
Let $s\ge 1$ and $\omega_1,\ldots,\omega_s\in (0,+\infty)$ be fixed independently of $N$, and set
$\mathcal S=\{\omega_1,\ldots,\omega_s\}.$ Let
$L_1,\ldots,L_N\subseteq \R^2$
be lattices satisfying
$\text{vol}(L_i)\in \mathcal S$, $1\le i\le N.$
Choose initial matrices $A_i\in \mathrm{GL}(2,\R)$ such that
$L_i=A_i\Z^2.$
Then, for every $\varepsilon>0$, there exist matrices
$
B_i\in \mathrm{GL}(2,\R),$
$C_i\in \mathrm{SL}(2,\Z),$ and a nonnegative function $f\in L^1(\R^2)$ such that, with
$\Lambda_i=B_i\mathbb Z^2$,
\[
\det B_i=\det A_i,
\quad
\|B_i-A_iC_i\|_\infty\le \varepsilon,
\quad
\Lambda_i\cap\Lambda_j=\{0\}\quad(i\ne j),
\]
and $f$ tiles with $\Lambda_1,\ldots,\Lambda_N$ and satisfies
\[
\diam\supp f=O(\sqrt N).
\]  
\end{theorem}
This means that given any lattices
$L_i$ we can find lattices $\Lambda_i$, such that there exists a basis of $L_i$ and a basis of $\Lambda_i$ that are
$\varepsilon$-close, and such that the $\Lambda_i$ have a common tiling function with small diameter.

\section{The Proof of Theorem \ref{1/d}}\label{sec2}
To prove Theorem \ref{1/d}, we establish and develop several auxiliary results. They concern, respectively, the preservation of the tiling property and the construction of lattices with pairwise trivial intersections.

\subsection{Tiling constructions}
Let $\mathbf 0_n=(0,\ldots,0)$ denote the zero vector in $\R^n$, and let
$$
\pi_2:\R^m\times\R^n\longrightarrow\R^n,
\quad
\pi_2(x,y)=y,
$$
be the canonical projection onto the second factor. We identify a
lattice $L\subseteq\R^m$ with
$L\times\{\mathbf 0_n\}\subseteq\R^{m+n}.$ The following lifting lemma allows us to pass from tilings in the factors to a tiling in the product.

\begin{lemma}\label{lm:projection}
Let $d=m+n$, $L \subseteq \R^m$, $\Gamma \subseteq \R^n$ be lattices. Assume $g:\R^m\to\R$ tiles with $L$ at level $c_1$ and $h:\R^n\to\R$ tiles with $\Gamma$ at level $c_2$. Assume also that $M'$ is a lattice of rank $n$ in $\R^d$ such that $\pi_2(M') = \Gamma$.

Then $f(x, y) = g(x) h(y)$ tiles with $L\times\{\mathbf 0_n\} +M'$ at level $c_1c_2$.
\end{lemma}

\begin{proof}
Since $\pi_2(M') = \Gamma$ and $M'$ and $\Gamma$ are lattices of the same rank $n$, the restriction $\pi_2|_{M'}$ maps $M'$ bijectively onto $\Gamma$. Hence
\[(L\times\{\mathbf 0_n\})\cap M' =\{0\}.\]
It follows that the sum $L\times\{\mathbf 0_n\}+M'$ is direct. Therefore,
\begin{align*}
\sum_{n' \in L\times\{\mathbf 0_n\}, m'=(m_1', m_2') \in M'} f((x, y)-n'-m') &= \sum_{\ell \in L, m' \in M'} f(x-\ell-m_1', y-m_2') \\
 &= \sum_{\ell \in L, m' \in M'} g(x-\ell-m_1') h(y-m_2')\\
 &= \sum_{m' \in M'} \sum_{\ell \in L} g(x-\ell-m_1') h(y-m_2') \\
 &=c_1\sum_{m' \in M'} h(y-m_2')\\
 &=c_1\sum_{m \in \Gamma} h(y-m)\\
 &= c_1 c_2.
\end{align*}
This shows that $f$ tiles with $L\times\{\mathbf 0_n\}+M'$ at level
$c_1c_2$, thereby completing the proof.
\end{proof}

\begin{corollary}\label{cor:skew}
Let $d=m+n$, $L = A\Z^m \subseteq \R^m$, $\Gamma = T\Z^n \subseteq \R^n$ be lattices. Assume $g:\R^m\to\R$ tiles with $L$ at level $c_1$ and $h:\R^n\to\R$ tiles with $\Gamma$ at level $c_2$. Define the matrix
$$
B = \begin{pmatrix} A & S \\ 0 & T \end{pmatrix}
$$
where $S\in \R^{m\times n}$. Then $\Lambda = B\Z^d$ is a lattice and $f(x, y) = g(x) h(y)$ tiles with $\Lambda$ at level $c_1c_2$.
\end{corollary}

\begin{proof}
Define the rank $n$ lattice $M' = \begin{pmatrix} S\\T\end{pmatrix} \Z^n$. Then $\Lambda = L\times\{\mathbf 0_n\}+M'$, $\pi_2(M') = \Gamma$ and the result follows from Lemma \ref{lm:projection}. This completes the proof.
\end{proof}

\begin{corollary}\label{cor:collection}
Let $d=m+n$, with $m, n>0$, $N\ge 1$ and let $L_i$, $i=1,\ldots,N$, be lattices in $\R^m$ and $g:\R^m\to\R$ be a common tiling function for the $L_i$. Let also $M'_i$ be rank $n=d-m$ lattices in $\R^d$ such that $\pi_2(M'_i) = c\Z^n$ with $c$ a nonzero constant. Then the function $f(x, y) = g(x) h(y)$ tiles with all $L_i\times\{\mathbf 0_n\} +M'_i$, where $h$ is any function that tiles with $c\Z^n$.
\end{corollary}

\begin{proof}
It follows immediately from Corollary \ref{cor:skew}.
\end{proof}

Recall that a nonzero function $f\in L^1(\R^d)$ tiles with each of the full-rank lattices $\Lambda_1,\ldots,\Lambda_N$ if and only if $\hat{f}$ vanishes on $\cup_{i=1}^N \Lambda_i^*\setminus\{0\}$. 
When considering common tiling functions for several lattices, one naturally examines the relations among their duals. The dual of a full‑rank lattice is a full‑rank lattice. The following result shows that when the lattices intersect only at $\{0\}$, the intersection of their duals is not full‑rank.

In this paper, if we say $x\in\R^d$, that is $x=(x_1,\ldots,x_d)^T$, where $x_i$ is the $i$-th coordinate of $x$.
\begin{lemma}\label{rank}
    For $d\ge 2$, assume that $\Lambda_1,\Lambda_2\subseteq \R^d$ are lattices
    of volume 1 satisfying $\Lambda_1\cap \Lambda_2=\{0\}$.
    Then $\rank(\Lambda_1^*\cap\Lambda_2^*)\le d-2$.
\end{lemma}
\begin{proof}
   Let
$\Lambda_i=A_i\Z^d,$
where $A_i\in \mathrm{GL}(d,\R),i=1,2.$
Suppose, to the contrary, that we have $\rank(\Lambda_1^*\cap \Lambda_2^*)> d-2$. 
It follows from 
\[\rank(\Lambda_1^*\cap \Lambda_2^*)=\rank(A_1^{T}\Lambda_1^*\cap A_1^{T}\Lambda_2^*)=\rank(\Z^d\cap A_1^{T}A_2^{-T}\Z^d)\]
that
\begin{equation}%\label{rank dual cap}
  \rank(\Z^d\cap C\Z^d)>d-2,\nonumber
\end{equation}
where $C=A_1^{T}A_2^{-T}$.
Choose $d-1$ $\Z$-linearly independent vectors $\ell_1,\dots,\ell_{d-1} \in \Z^d\cap C\Z^d$. Since $\ell_i \in C\Z^d$ for $1\le i\le d-1$,  $k_i:=C^{-1}\ell_i\in \Z^d$, and consequently 
$k_i,\ell_i\in  \Z^d$ for all i.

Given $d-1$ vectors $v_1,\ldots,v_{d-1}$ in $\R^d$, let $\mathcal{V}(v_1,\ldots,v_{d-1})\in \R^d$ be defined by
\[\left \langle \mathcal{V}(v_1,\ldots,v_{d-1}), x  \right \rangle=\det(v_1,\ldots,v_{d-1}, x), \quad \forall x\in \R^d.  \]
Let $U=(\ell_1,\ldots,\ell_{d-1})$ be a $d\times(d-1)$ matrix and $U_k$ denote the $(d-1)\times(d-1)$ matrix obtained from $U$ by deleting $k$-th row. Then we have
\[\left \langle \mathcal{V}(\ell_1,\ldots,\ell_{d-1}), x  \right \rangle=\det (\ell_1,\ldots,\ell_{d-1},x)=\sum_{k=1}^{d}(-1)^{k+d}x_k \det(U_k).\]
Combining this with 
\[\left \langle \mathcal{V}(\ell_1,\ldots,\ell_{d-1}), x  \right \rangle=\sum_{k=1}^{d}\mathcal{V}_k(\ell_1,\ldots,\ell_{d-1})x_k,\]
we obtain
$\mathcal{V}_k(\ell_1,\ldots,\ell_{d-1})=(-1)^{k+d}\det (U_k)$,
where $\mathcal{V}_k(\ell_1,\ldots,\ell_{d-1})$ is the 
$k$-th coordinate of the vector $\mathcal{V}(\ell_1,\ldots,\ell_{d-1})$. 
Since $\rank(U)=d-1$, there exists some k such that $\det(U_k)\ne 0$, which implies that $\mathcal{V}_k(\ell_1,\ldots,\ell_{d-1})\ne 0$. Hence $\mathcal{V}(\ell_1,\ldots,\ell_{d-1})\ne 0$.
Moreover, since each $\ell_i\in\Z^d$, it follows that
$U_k\in \mathrm{M}_{(d-1)\times(d-1)}(\Z)$ for each $k$.
Thus $\det(U_k)\in\Z$ for each $k$, and hence
\begin{equation}\label{v-coordinate}
0\ne \mathcal{V}(\ell_1,\ldots,\ell_{d-1})\in\Z^d.
\end{equation}
Similarly, we have 
\begin{equation}\label{v-k-coordinate}
0\ne\mathcal{V}(k_1,\ldots,k_{d-1})\in \Z ^d.  
\end{equation}

For any $x\in \R^d$, we have
\begin{align*}
    \left \langle \mathcal{V}(\ell_1,\ldots,\ell_{d-1}), x  \right \rangle&=\det(\ell_1,\ldots,\ell_{d-1},x)\\
    &=\det(Ck_1,\ldots,Ck_{d-1},C\cdot C^{-1}x)\\
    &=\det(C)\det(k_1,\ldots,k_{d-1},C^{-1}x)\\
    &=\det(C)\left \langle C^{-T}\mathcal{V}(k_1,\ldots,k_{d-1}), x  \right \rangle.
\end{align*}
Since $A_1$ is invertible, the assumption
$\Lambda_1\cap\Lambda_2=\{0\}$ is equivalent to
\[
A_1^{-1}\Lambda_1\cap A_1^{-1}\Lambda_2
=\Z^d\cap A_1^{-1}A_2\Z^d
=\{0\}.
\]
Moreover, since 
\[|\det C|=\left|\frac{\det A_1}{\det A_2}\right|=\frac{\text{vol}(\Lambda_1)}{\text{vol}(\Lambda_2)}=1,\] 
it follows from \eqref{v-k-coordinate} that
\[0\ne\mathcal{V}(\ell_1,\ldots,\ell_{d-1})=\det C \cdot C^{-T}\mathcal{V}(k_1,\ldots,k_{d-1})\in C^{-T}\Z^{d}=A_1^{-1}A_2\Z^d.\]
Together with \eqref{v-coordinate}, this gives
\[
0\ne \mathcal{V}(\ell_1,\ldots,\ell_{d-1})
\in
\Z^d\cap A_1^{-1}A_2\Z^d,
\]
which is a contradiction. Hence  $\rank(\Lambda_1^*\cap\Lambda_2^* )\le d-2$, and the proof is complete.
\end{proof}

From a rank-theoretic viewpoint, the quantity
\[d-\rank (\Lambda_1^*\cap\Lambda_2^*)\]
measures the number of residual frequency directions not accounted for by the common dual subgroup. Thus, a larger intersection rank leaves fewer directions in which the two tiling conditions must be treated independently, and consequently is more favorable for constructing a common tiling function with small support. In the extremal case
\[\rank (\Lambda_1^*\cap\Lambda_2^*)=d-2\]
only two residual directions remain. This naturally motivates the following study of common tiling functions for plane lattices.

\begin{lemma}\label{tile function}
For $i=1,\ldots,N$ let
\[B_i=\begin{pmatrix}
a_i& b\\ 
c_i& d_i
\end{pmatrix}\in \mathrm{GL}(2,\R)\]
satisfy $|\det B_i|=1$ and $b\ne 0$. Let $\Lambda_i=B_i\Z^2$ and $v_i=(b,d_i)^T$. For each i, choose $w_i\in \R^2$ linearly independent of $v_i$ and define
$P_i=[0,1)v_i+[0,1)w_i$.
Then the function $\tau$ defined by
\[
\tau=|b|^2\mathds 1_{\left[0,1/|b|\right)^2}\ast
\frac{1}{|\det(v_1,w_1)|}\mathds 1_{P_1}\ast\cdots \ast\frac{1}{|\det(v_N,w_N)|}\mathds 1_{P_N}
\]
belongs to $L^1(\R^2)$ and tiles with all lattices $\Lambda_i$.  
\end{lemma} 
\begin{proof}
First, we observe that $\tau\in L^1(\R^2)$. Indeed, the function
$|b|^2\mathds 1_{[0,1/|b|)^2}$ has $L^1$-norm equal to $1$. Moreover, for each
$i=1,\ldots,N$, since $w_i$ is linearly independent of $v_i$, we have
$\det(v_i,w_i)\neq 0$ and
\[
\left\|\frac{1}{|\det(v_i,w_i)|}\mathds 1_{P_i}\right\|_1
=
\frac{|P_i|}{|\det(v_i,w_i)|}
=1.
\]
Thus all factors in the convolution defining $\tau$ belong to $L^1(\R^2)$, and hence
$\tau\in L^1(\R^2)$ by Young's inequality. 

Computing directly, we obtain the dual lattice of $\Lambda_i$ as
\[\Lambda_i^{\ast}=B_i^{-T}\Z^2
=\begin{pmatrix}
d_i  &  -c_i\\
-b  & a_i
\end{pmatrix}\Z^2.\]
Fix $i$. For any $\xi\in \Lambda_i^{\ast}$, there exist $(p,q)^T\in\Z^2$ such that
\[\xi=B_i^{-T}(p,q)^{T}=(d_ip-c_iq,-bp+a_iq)^T,\]
which, combined with the assumptions, implies
\[\left |\left \langle \xi , v_i  \right \rangle\right|=\left |q(a_id_i-bc_i) \right|=\left | q \right |\in \Z.\]
Thus, for all $(p,q)^T\in \Z^{2}$,
\begin{equation}\label{integer1}
\left |\left \langle B_i^{-T}(p,q)^{T} , v_i  \right \rangle \right|=\left | q \right |.
\end{equation}

For any $\xi\in \Lambda_i^{*}$, we have
\[\widehat{\mathds 1_{P_i}}(\xi)
=\int_{\R^2}\mathds 1_{P_i}(x)e^{-2\pi i \left \langle \xi, x  \right \rangle}dx
=|\det(v_i,w_i)|\int_{0}^{1}e^{-2\pi i k\left \langle \xi, v_i \right \rangle}dk\int_{0}^{1}e^{-2\pi i \ell\left \langle \xi, w_i \right \rangle}d\ell.\]
Combining this with \eqref{integer1}, if $q\in \Z\setminus\{0\}$, then
\begin{equation}\label{zero-i1}
    \widehat{\mathds 1_{P_i}}(\xi)=|\det(v_i,w_i)|\cdot  0\cdot \int_{0}^{1}e^{-2\pi i \ell\left \langle \xi, w_i \right \rangle}d\ell=0.
\end{equation}
A direct calculation shows that the zero set of the Fourier transform of $\mathds{1}_{[0,1/|b|)^2}$ is given by
\begin{equation}\label{cube-zero1}
\left\{
\xi\in\R^2:
\xi_j\in b\Z\setminus\{0\}
\text{ for some } j=1,2
\right\}.
\end{equation}

Now let
$\xi\in\cup_{i=1}^{N}\Lambda_i^{\ast}$. Then
\[\xi=B_i^{-T}(p,q)^{T}=(d_ip-c_iq,-bp+a_iq)^{T}\]
for some i and some $(p,q)^T\in\Z^2$. If $\xi\ne 0$, we have $(p,q)\ne(0,0)$.
If $q\ne 0$, by (\ref{integer1}) and (\ref{zero-i1}), we have
$\widehat{\mathds 1_{P_i}}(\xi)=0$. Otherwise, by (\ref{cube-zero1}), we have $\widehat{\mathds{1}_{\left[0,1/|b|\right)^2}}(\xi)=0$.
It then follows from \eqref{tile-all-lattices} that $\tau$ tiles with all lattices $\Lambda_1,\ldots,\Lambda_N$. 
\end{proof}

\subsection{Lattices with pairwise trivial intersections}
We now turn to the auxiliary results concerning pairwise trivial intersections.

\begin{lemma}\label{lattice-trival-disjoint}
Let $m,n$ be positive integers and let $A_1,\ldots,A_N\in \mathrm{GL}(m,\R)$ satisfy 
\[
A_i\Z^m \cap A_j\Z^m = \{0\}, \quad \text{for all } i \neq j.
\]
Then, for any $T\in \mathrm{GL}(n,\R)$, there exist 
$S_1,\dots,S_N \in \mathrm{M}_{m\times n}(\R)$ such that the lattices
\[
\Lambda_i = \begin{pmatrix} A_i & S_i \\ 0 & T \end{pmatrix} \Z^{m+n}, \quad i=1,\dots,N,
\]
satisfy $\Lambda_i \cap \Lambda_j = \{0\}$ for all $i \neq j$.
\end{lemma}

\begin{proof}
Notice that
\[\Lambda_i=\begin{pmatrix}
A_i & S_i\\
 0 & T
\end{pmatrix}\Z^{m+n}=\left\{\left(a_i+S_i k, Tk\right)^T\in \R^{m+n}:a_i\in A_i\Z^m,k\in \Z^{n}\right\}.\]
We select
the matrices $S_i$ inductively: once $S_1,\ldots,S_{i-1}$ have
been fixed, $S_i$ is chosen so that
\[
\Lambda_i\cap \Lambda_j=\{0\},
\quad\text{for every }j<i.
\]
Proceeding in this way yields $N$ lattices satisfying
\[
\Lambda_i\cap \Lambda_j=\{0\},
\quad\text{for all }i\neq j.
\]

Suppose that $S_1,\ldots,S_{i-1}$ have already been chosen. Fix $j<i$. If $x\in \Lambda_i\cap \Lambda_j$, there exists $a_i\in A_i\Z^m,a_j\in A_j\Z^m,k,k'\in\Z^n$ such that 
\[(a_i+S_i k,Tk)=(a_j+S_j k',Tk'),\]
and consequently $k=k'$. If $k=0$, then $x=(a_i,0)=(a_j,0)$. It follows from $A_i\Z^m\cap A_j\Z^m=\{0\}$ that $x=0$. Otherwise, we have $x\ne 0$ and 
\begin{equation}%\label{bad case}
   S_ik=S_jk+a_j-a_i.\nonumber
\end{equation}

Define
\[
\mathcal{S}_{j,k,a_i,a_j}=\left\{S\in \mathrm{M}_{m\times n}(\R): Sk=S_jk+a_j-a_i\right\},\]
which is an undesirable set, because it would cause $\Lambda_i$ and $\Lambda_j$ to intersect nontrivially. If $ Sk=S_jk+a_j-a_i$, which equivalent to 
\[u_tk=\left(S_jk+a_j-a_i\right)_t,\quad t=1,\ldots,m,\]
where $S=(u_1,\ldots,u_m)^{T}, u_t^{T}\in \R^n,(S_jk+a_j-a_i)_t$ is the $t$-th coordinate of $S_jk+a_j-a_i$. Combining this with $k\ne 0$, we obtain that
$u_t$ has an $(n-1)$-dimensional solution space for all $1\le t\le m$. Then the dimension of $\mathcal{S}_{j,k,a_i,a_j}$ is $mn-m$. Combining this with $\mathrm{M}_{m\times n}(\R)\cong \R^{mn}$ yields that $\mathcal{S}_{j,k,a_i,a_j}$ is a null set with respect to the Lebesgue measure on $\R^{mn}$. Hence
\[\mathcal{W}_i=\bigcup_{j< i}\bigcup_{k\in \Z^n\setminus  \{0\}}\bigcup_{a_i,a_j\in \Z^m }\mathcal{S}_{j,k,a_i,a_j}\]
is also a null set with respect to the Lebesgue measure on $\R^{mn}$. Therefore, we may choose $S_i\in \mathrm{M}_{m\times n}(\R)\setminus\mathcal{W}_i$, thereby completing the proof.    
\end{proof}

A direct argument yields the following simple observation.
\begin{lemma}\label{upper-matrix-disjoint}
For $i=1,\ldots,N$ let
\[B_i=\begin{pmatrix}
b_{11}^{(i)}& b_{12}^{(i)}&\cdots &b_{1d}^{(i)}\\ 
0& b_{22}^{(i)}&\cdots & b_{2d}^{(i)}\\
\vdots & \ddots & \ddots & \vdots\\
0 & \cdots & 0 & b_{dd}^{(i)}
\end{pmatrix}\in \mathrm{GL}(d,\R),\]
and define $\Lambda_i=B_i\Z^d$. Suppose that, for $r=1,\ldots,d,$ 
\[b_{rr}^{(i)}/b_{rr}^{(j)} \in \R\setminus\Q,\quad \text{whenever}\quad i\ne j.\]
Then $\Lambda_i\cap \Lambda_j=\{0\}$ for all $i\neq j$.
\end{lemma}
\begin{proof}
Fix two distinct indices $i\ne j$ and take any $x\in \Lambda_i\cap \Lambda_j$. By definition of the lattices, there exist integer vectors 
\[p=(p_1,\ldots,p_d)^T, \quad q=(q_1,\ldots,q_d)^T\in\Z^d\]
such that $x=B_ip=B_jq$. Since $B_i$ and $B_j$ are invertible upper triangular matrices, their
diagonal entries are nonzero, and the $d$-th coordinate of $B_ip=B_jq$
satisfies
\begin{equation}\label{d-coordinate}
 b_{dd}^{(i)}p_d=b_{dd}^{(j)}q_d.
\end{equation}
It follows that $p_d=0$ if and only if
$q_d=0$. If $p_d\neq 0$, then also
$q_d\neq 0$, and \eqref{d-coordinate} gives
\[
b_{dd}^{(i)}/b_{dd}^{(j)}=q_d/p_d\in\Q,
\]
contradicting the hypothesis that $b_{dd}^{(i)}/b_{dd}^{(j)}$ is irrational.
Therefore,
$p_d=q_d=0.$

Now we proceed by descending induction. Suppose that for some $k$ with $1\le k<d$, we have already shown that
\[p_{k+1}=\cdots=p_{d}=0,\quad q_{k+1}=\cdots=q_d=0.\]
Equating the $k$-th coordinates of $B_ip$ and $B_jq$ yields
\[
\sum_{r=k}^{d}b_{kr}^{(i)}p_r=\sum_{r=k}^{d}b_{kr}^{(j)}q_r.\]
By the induction hypothesis, all terms with $r>k$ vanish, so this reduces to
\[
b_{kk}^{(i)}p_k=b_{kk}^{(j)}q_k .
\]
By the same argument as for the $d$-th coordinate, using the nonvanishing of the diagonal entries and the irrationality of $b_{kk}^{(i)}/b_{kk}^{(j)}$, we obtain
$p_k=q_k=0.$ By induction, we conclude that $p=q=0$. Hence $x=0$, and therefore $\Lambda_i\cap \Lambda_j=\{0\}$. This completes the proof.
\end{proof}

We shall use the following notation throughout the remainder of the paper.
For positive quantities $a$ and $b$, we write $a\lesssim b$ if
$a\leq Cb$ for some implicit constant $C>0$. We write $a\asymp b$ if
both $a\lesssim b$ and $b\lesssim a$ hold.

We are now ready to prove Theorem \ref{1/d}.

\begin{proof}[Proof of Theorem \ref{1/d}.]
We divide the proof into two cases. The first case is when $d=2$. In the second case, we use the two-dimensional construction to handle the case $d>2$, which completes the proof.

\textbf{Case 1}: $d=2$.

Fix $s,t\in \R$ with $0<s<t$. We choose $a_1,\ldots,a_N\in\left[s\sqrt{N},t\sqrt{N}\right]$ inductively.
After $a_1,\ldots,a_{i-1}$ have been chosen, the forbidden set 
\[\bigcup_{j<i}a_j\Q\]
is countable, while $\left[s\sqrt{N},t\sqrt{N}\right]$ is uncountable. Hence we may choose
\[
a_i\in \left[s\sqrt N,t\sqrt N\right]\setminus \bigcup_{j<i}a_j\Q.
\]
This gives
\[
a_i/a_j\notin\Q,
\quad\text{whenever }i\neq j.
\]
Let 
\[B_i=\begin{pmatrix}
a_i& b\\ 
0& 1/a_i
\end{pmatrix},\quad\Lambda_i=B_i\Z^2,\quad v_i= (b,1/a_i)^T,
\]
where $b$ is a nonzero constant satisfying $|b| \asymp  \frac{1}{\sqrt{N}}$. 
By Lemma \ref{upper-matrix-disjoint}, we obtain 
\[\Lambda_i\cap\Lambda_j=\{0\}, \quad  \text{for  all }\ i\ne j.\]
It is straightforward to verify that $\text{vol}(\Lambda_i)=1$. 
Choose $w_i\in \R^2$ linearly independent of $v_i$ with $\left \| w_i \right \|\lesssim \frac{1}{\sqrt{N}}$, and let
\[
P_i=\{kv_i+\ell w_i:0\leq k< 1,0\leq \ell< 1\}.\]

Define
\[
\tau=|b|^2\mathds 1_{\left[0,1/|b|\right)^2}\ast
\frac{1}{|\det(v_1,w_1)|}\mathds 1_{P_1}\ast\cdots \ast
\frac{1}{|\det(v_N,w_N)|}\mathds 1_{P_N}.
\]
It follows from Lemma \ref{tile function} that $\tau\in L^1(\R^2)$ and that
$\tau$ tiles with all lattices $\Lambda_1,\ldots,\Lambda_N$.
By a direct calculation, we obtain
\[\diam \supp \mathds 1_{\left[0,1/|b|\right)^2}\lesssim \sqrt{N}\]
and 
\[\diam \supp \mathds1_{P_i}\le \left \| v_i \right \|+\left \| w_i \right \|\lesssim\frac{1}{\sqrt N}.\]
Therefore, 
\[\diam \supp \tau\le \diam \supp\mathds 1_{\left[0,1/|b|\right)^2} +\sum_{i=1}^{N}\diam\supp \mathds1_{P_i}\lesssim \sqrt N+N\cdot\frac{1}{\sqrt N}\lesssim \sqrt N.\]
%Combining this with \eqref{lower bound} yields $\diam \supp \tau=O(\sqrt{N})$. 
Thus, in the case $d=2$, we take $f:=\tau$, which completes the construction.

\textbf{Case 2}: $d>2$.

Let $\alpha =N^{\frac{2-d}{2d}}$, $\beta=N^{\frac{1}{d}}$, and 
let $B_1,\ldots,B_N\in \mathrm{GL}(2,\R)$ be the
matrices constructed in Case~1. By construction, the corresponding
plane lattices satisfy
\[
B_i\Z^2\cap B_j\Z^2=\{0\},\quad
\text{whenever}\  i\neq j,\] and $\det B_i=1$ for all i. Let $n=d-2$. For $1\le i\le N$ and $S_i\in \R^{2\times n}$, consider
\[
\Lambda_i=
\begin{pmatrix}
\alpha B_i & S_i\\
0 & \beta I_n
\end{pmatrix}\Z^d,
\]
where $I_n$ denotes the $n\times n$ identity matrix. Independently of the choice of $S_i$,
\[
\operatorname{vol}(\Lambda_i)
=\alpha^2\beta^{d-2}
=1.
\]
By Lemma \ref{lattice-trival-disjoint}, we may choose $S_1,\ldots,S_N$ so that
\[
\Lambda_i\cap\Lambda_j=\{0\},
\quad\text{for all }i\neq j.
\]
We fix such a choice of $S_1,\ldots,S_N$ for the remainder of the argument.

We now construct a nonnegative function $f\in L^1(\R^d)$ that tiles with all $\Lambda_i$.
Let 
\[g(x)=\alpha^{-2}\tau(x/\alpha), \quad h(y)=\beta^{-n}\mathds{1}_{[0,\beta)^{n}}(y),\quad f(x,y)=g(x)h(y).\]
Since $g\in L^1(\R^2)$ and $h\in L^1(\R^n)$, it follows immediately that $f\in L^1(\R^{d})$.
For $g(x)=\alpha^{-2}\tau(x/\alpha)$, we have
$\widehat g(x)=\widehat\tau(\alpha x).$ For each $i=1,\ldots,N$, since $\tau$ tiles with
$B_i\Z^2$, by \eqref{tile-all-lattices},  $\widehat\tau$ vanishes on 
\[\bigcup_{i=1}^{N}B_i^{-T}\Z^2\setminus\{0\}.\]
Thus $\widehat g$ vanishes on
\[
\bigcup_{i=1}^{N}\alpha^{-1}B_i^{-T}\Z^2\setminus\{0\}=\bigcup_{i=1}^{N}
(\alpha B_i)^{-T}\Z^2\setminus\{0\}.
\]
Applying \eqref{tile-all-lattices} again, we conclude that $g$ tiles with all
lattices $\alpha B_1\Z^2,\ldots,\alpha B_N\Z^2$. Moreover, the zero set of $\widehat h$ is given by
\[
\left\{\xi\in\R^n:
\xi_i\in \beta^{-1}\Z\setminus\{0\}
\text{ for some } i=1,\ldots,n
\right\}.
\]
Since $(\beta I_n)^{-T}\Z^n=\beta^{-1}\Z^n,$
$\widehat h$ vanishes on
$(\beta I_n)^{-T}\Z^n\setminus\{0\}.$
Hence, by \eqref{tile-all-lattices}, $h$ tiles with $\beta I_n\Z^n$. Thus, by Corollary \ref{cor:collection}, $f$ tiles with all lattices
$\Lambda_1,\ldots,\Lambda_N$. We now estimate the diameter of $\diam \supp f$. We have
\[\diam \supp g=\alpha\ \diam \supp \tau\lesssim\alpha\sqrt{N}=N^{1/d}\]
and
\[
\diam \supp h=\sqrt{n}\beta
\lesssim N^{1/d}
.\]
Consequently,
\[\diam \supp f \le \diam \supp g+\diam \supp h\lesssim N^{1/d}.\]
%Combining this with \eqref{lower bound}, this gives $\diam \supp f=O(N^{1/d})$. 
This shows that $\diam\supp f=O(N^{1/d})$, thereby completing the proof.
\end{proof}

\section{The Proof of Theorem \ref{-close}}
\begin{defn}
Let $\Lambda$ be a lattice. A nonzero vector  $v\in \Lambda$ is called primitive if for any integer $n\ge 2$ and $u\in \Lambda$,
$v\ne nu.$
\end{defn}

The following lemma fixes a primitive vector of the lattice and partitions the dual lattice into two disjoint subsets with respect to this vector.

\begin{lemma}\label{disjoint}
Let $c>0$ and let $\Lambda$ be a lattice of volume c in $\R^2$. If $v=(v_1,v_2)^T$ is a primitive vector of $\Lambda$, then 
\[\Lambda^*=\left\{\xi\in\Lambda^*:\left \langle \xi, v \right \rangle\in\Z\setminus\{0\} \right\}\cup \frac{\Z}{c}(-v_2,v_1)^T.\]
\end{lemma}
\begin{proof}
It follows from $v\in \Lambda$ that we can decompose $\Lambda^*$ into two disjoint sets with respect to $v$,
\[\Lambda^*=\{\xi\in \Lambda^*:\left \langle \xi, v \right \rangle\in\Z\setminus\{0\} \}\cup \{\xi\in \Lambda^*:\left \langle \xi, v \right \rangle=0 \}.\]
So, we only need to show that
\[\frac{\Z}{c}(-v_2,v_1)^T=\{\xi\in \Lambda^*:\left \langle \xi, v \right \rangle=0 \}.\]
Since $v$ is a primitive vector of $\Lambda$, by 
\cite[Theorem 3.1.14]{hans2024},
there exists $u\in \Lambda$ such that 
\[\Lambda=\Z v+\Z u.\]
Then we obtain
\[\Lambda^*=\frac{1}{c}\begin{pmatrix}
    u_2 & -v_2\\ 
 -u_1 & v_1
\end{pmatrix}\Z^2.\]
Let $\ell_1=\frac{1}{c}(u_2,-u_1)^T$ and $\ell_2=\frac{1}{c}(-v_2,v_1)^T$.
A direct calculation yields
\[\left | \left \langle \ell_1, v \right \rangle\right | =1,\quad \left \langle\ell_1, u\right \rangle=0,\quad \left \langle\ell_2, v\right \rangle=0,\quad \left |  \left \langle\ell_2, u\right \rangle\right | =1.
\]
Together with the definition of the dual lattice, this yields
\[\frac{\Z}{c}(-v_2,v_1)^T\subset\{\xi\in \Lambda^*:\left \langle \xi, v \right \rangle=0 \}.\]

Conversely, take any $\xi\in\Lambda^*$ with $\langle \xi, v \rangle=0$.
Since $\xi$ can be expressed uniquely as $\xi=n\ell_1+m\ell_2$ for some $n,m\in\Z$, we have
\[\left \langle \xi, v \right \rangle=\left \langle (n\ell_1+m\ell_2), v\right \rangle=\left \langle n\ell_1, v\right \rangle+\left \langle m\ell_2, v\right \rangle=0.\]
From $\left |\left \langle \ell_1, v\right \rangle\right |=1$ and $\left \langle\ell_2, v\right \rangle=0$ we immediately get $n=0$. Consequently,
\[\xi=m\ell_2=\frac{m}{c}(-v_2,v_1)^T\in\frac{\Z}{c}(-v_2,v_1)^T.\]
This shows the reverse inclusion
\[
\frac{\Z}{c}(-v_2,v_1)^T\supset\left\{\xi\in \Lambda^*:\langle \xi, v \rangle=0\right\}.
\]
Therefore, we conclude that
\[\Lambda^*=\left\{\xi\in\Lambda^*:\left \langle \xi, v \right \rangle\in\Z\setminus\{0\} \right\}\cup \frac{\Z}{c}(-v_2,v_1)^T.\]
\end{proof}

Before proceeding to the proof, we introduce the following notation. For $A=(a_{ij})\in \R^{d\times d}$, define
\[
\left \|  A\right \|_{\infty}:=\max_{1\leq i,j\leq d}|a_{ij}|.
\]

We are now equipped to prove Theorem \ref{-close}.

\begin{proof}[Proof of Theorem \ref{-close}.] 
Set $\text{vol}(L_i)=\omega_{j_i}\in \mathcal S.$ The proof proceeds in three steps.

\textbf{Step 1}: Since
\[
L_i=A_i\Z^2=A_iC_i\Z^2,
\quad
\text{for every }C_i\in \mathrm{SL}(2,\Z),
\]
we construct $C_i(k'_i,\ell_i)\in \mathrm{SL}(2,\Z)$, where $k'_i\in\Z\setminus\{0\}$ and $\ell_i\in\Z$ are parameters. 
The parameters $k'_i$ and $\ell_i$ are then determined in Step 2: we choose $k'_i$ so that
$
\left\|B_i-A_iC_i(k'_i,\ell_i)\right\|_{\infty}\leq\varepsilon,$
where $B_i$, constructed below, is a basis matrix whose columns generate the desired lattice $\Lambda_i$.
The parameter $\ell_i$ is chosen so as to
control the norms of a specific primitive vector in $\Lambda_i$.
The same construction is used for every $i$. For notational simplicity, we write
$
C_i:=C_i(k'_i,\ell_i).
$

Since
\[
A^{-1}_i\{(x,y)^T:x>0,\ |y|<x\}
\]
is a nonempty open cone, we can choose $x_i=(x_{i_1},x_{i_2})^T\in\Z^2$ with $\gcd(x_{i_1},x_{i_2})=1$ such that $a_i=A_ix_i=(a_{i_1},a_{i_2})^T$ satisfies
\begin{equation}\label{ai1ai2}
    a_{i_1}>0,\quad |a_{i_2}|\le a_{i_1}.
\end{equation}
It follows from the Euclidean algorithm that there exists $y_i\in \Z^2$ such that 
\[\det(y_i, x_i)=1.\]
Let $k'_i\in\Z\setminus\{0\}$ and $\ell_i\in\Z$ be parameters. Define
\[k_i=(k'_i,1)^T,\quad C_i=(C_{i_1},C_{i_2}),\]
where 
\begin{equation}%\label{Ci1Ci2}
   C_{i_1}=y_i+\ell_i x_i, \quad
   C_{i_2}=x_i-k'_iC_{i_1}. \nonumber
\end{equation}
Then we obtain
\begin{equation}%\label{xi}
C_ik_i=k'_iC_{i_1}+C_{i_2}=x_i,\nonumber
\end{equation}
and hence 
\begin{equation}\label{ai}
    A_iC_ik_i=A_ix_i=a_i .
\end{equation}
Moreover, from the construction we see that $C_i \in \Z^{2\times 2}$, 
and the calculation
\[\det C_i=\det (C_{i_1},C_{i_2})=\det (C_{i_1},x_i)=\det (y_i,x_i)=1,\]
then yields $C_i \in \mathrm{SL}(2,\Z)$.

\textbf{Step 2}: Fix $r\in \R$ with $0<|r|\asymp \frac{1}{\sqrt{N}}$. For each i, let 
\[v(d_i)=(r,d_i)^T,\quad \Lambda_i=B(d_i,t_i)\Z^2,\] where $d_i$ and $t_i$ are parameters to be chosen. We first construct $B(d_i,t_i)\in \mathrm{GL}(2,\R)$ with $d_i$ and $t_i$ as parameters such that
\[B(d_i,t_i)k_i=v(d_i),\quad \det B(d_i,t_i)=\det A_i. \]
Then we determine the parameters $\{d_i\}_{i=1}^{N}$ and $\{t_i\}_{i=1}^{N}$  so that
\[\quad \left \|  B(d_i,t_i)-A_iC_i\right \|_{\infty}\le \varepsilon,\quad \text{for all}\ i,\]
and
\[\Lambda_i\cap \Lambda_j=\{0\}, \quad \text{for all}\ i\ne j.\] 

Define
\begin{equation}\label{B}
 B(d_i,t_i)=(B_1(d_i,t_i),B_2(d_i,t_i)),   
\end{equation}
\begin{equation}\label{tau}
\tau(d_i)=A_iC_{i_1}+\frac{v(d_i)-a_i}{k'_i} 
\end{equation}
and
\[\psi(d_i)=\frac{\det(\tau(d_i),v(d_i))-\det A_i}{r},\]
where 
\begin{equation}\label{Bi}
 B_1(d_i,t_i)=\tau(d_i)+(0,\psi(d_i))^T+t_i v(d_i),\quad B_2(d_i,t_i)=v(d_i)-k'_iB_1(d_i,t_i). 
\end{equation}
By a direct calculation, we have
\[B(d_i,t_i)k_i=k'_iB_1(d_i,t_i)+B_2(d_i,t_i)=v(d_i).\] 
Combining this with the identity $\gcd(k'_i,1)=1$ and Theorem 3.1.15 in
 \cite{hans2024}, we obtain that $v(d_i)$ is a primitive vector of $\Lambda_i$ for any $d_i$. 
A direct calculation yields
\begin{align}
\det B(d_i,t_i)&=\det(B_1(d_i,t_i), v(d_i))\nonumber\\
&=\det (\tau(d_i)+(0,\psi(d_i))^T,v(d_i))\nonumber\\
&=\det(\tau(d_i),v(d_i))+\det((0,\psi(d_i))^T,v(d_i))\nonumber\\
&=\det(\tau(d_i),v(d_i))-(\det(\tau(d_i),v(d_i))-\det A_i)\nonumber\\
&=\det A_i\nonumber.    
\end{align}

In the following, we shall determine the parameters $\{d_i\}_{i=1}^{N}$ and $\{t_i\}_{i=1}^{N}$  so that
\[\quad \left \|  B(d_i,t_i)-A_iC_i\right \|_{\infty}\le \varepsilon,\quad \text{for all}\ i,\]
and
\[\Lambda_i\cap \Lambda_j=\{0\}, \quad \text{for all}\ i\ne j.\] 

Let 
\begin{equation}\label{d_{i,0}-definition}
d_{i,0}=\frac{\det A_i+r\left((A_iC_{i_1})_2-\frac{a_{i_2}}{k'_i}\right)}{(A_iC_{i_1})_1-\frac{a_{i_1}}{k'_i}},    
\end{equation}
which depends on the parameters $k_i'$ and $\ell_i$.
Notice that
\[\det(A_iC_{i_1},A_ix_i)=\det A_i\det(C_{i_1},x_i)=\det A_i\det(y_i+\ell_ix_i,x_i)=\det A_i.\] 
Combining this with 
$\eqref{ai}$ yields
\[(A_iC_{i_1})_2=\frac{(A_iC_{i_1})_1  a_{i_2}-\det A_i}{a_{i_1}},\]
and hence
\[
d_{i,0}=r\frac{a_{i_2}}{a_{i_1}}+\frac{\det A_i-\frac{r\det A_i}{a_{i_1}}}{(A_iC_{i_1})_1-\frac{a_{i_1}}{k'_i}}=r\frac{a_{i_2}}{a_{i_1}}+\frac{\det A_i\left(1-\frac{r}{a_{i_1}}\right)}{(A_iy_i)_1+\ell_ia_{i_1}-\frac{a_{i_1}}{k'_i}}.    
\]
Choose $\ell_i$ so large that
\[
\ell_i\ge \left\lfloor
\frac{\frac{|\det A_i|}{|r|}\left|1-\frac{r}{a_{i_1}}\right|
+\left|(A_iy_i)_1\right|+a_{i_1}}{a_{i_1}}
\right\rfloor+1.
\]
Then
\begin{align*}
\left|(A_iC_{i_1})_1-\frac{a_{i_1}}{k'_i}\right|&=\left|
(A_iy_i)_1+\ell_i a_{i_1}-\frac{a_{i_1}}{k'_i}
\right|\\
&\ge
\ell_i a_{i_1}-\left|(A_iy_i)_1\right|-a_{i_1}\\
&>
\frac{|\det A_i|}{|r|}
\left|1-\frac{r}{a_{i_1}}\right|.    
\end{align*}
Thus the denominator of $d_{i,0}$ in \eqref{d_{i,0}-definition} is nonzero. Moreover,
\[
\left|
\frac{\det A_i\left(1-\frac{r}{a_{i_1}}\right)}
{(A_iy_i)_1+\ell_i a_{i_1}-\frac{a_{i_1}}{k'_i}}
\right|
\le |r|.
\]
Therefore, combining this with 
$\eqref{ai1ai2}$ yields
\begin{equation}\label{d_{i,0}}
|d_{i,0}|
\le
\left|r\frac{a_{i_2}}{a_{i_1}}\right|+|r|\le 2|r|\lesssim \frac{1}{\sqrt{N}}.
\end{equation}

Independently of the choice of $k_i'$ and $\ell_i$, we obtain
\begin{align}\label{1}
\det(\tau(d_{i,0}),v(d_{i,0}))&=\det\begin{pmatrix}
 (A_iC_{i_1})_1+\frac{r-a_{i_1}}{k'_i} & r    \\
 (A_iC_{i_1})_2+\frac{d_{i,0}-a_{i_2}}{k'_i} & d_{i,0}
    \end{pmatrix}\nonumber\\
&=d_{i,0}\left((A_iC_{i_1})_1-\frac{a_{i_1}}{k'_i}\right)-r\left((A_iC_{i_1})_2-\frac{a_{i_2}}{k'_i}\right)\nonumber\\
&=\det A_i+r\left((A_iC_{i_1})_2-\frac{a_{i_2}}{k'_i}\right)-r\left((A_iC_{i_1})_2-\frac{a_{i_2}}{k'_i}\right)(\text{By} \ \eqref{d_{i,0}-definition}.)\nonumber\\
&=\det A_i\nonumber
\end{align}
%It follows from $\eqref{Ci1Ci2}$ and $\eqref{ai}$ that 
%\begin{equation}\label{AiCi1-1}
% (A_iC_{i_1})_1=(A_iy_i)_1+\ell_ia_{i_1},\quad (A_iC_{i_1})_2=(A_iy_i)_2+\ell_ia_{i_2},  
%\end{equation}
%and hence, for any $k_i'\in \Z\setminus\{0\}$ and $\ell_i\in\Z$, we have
and hence
\begin{equation}\label{psidi0=0}
\psi(d_{i,0})=\frac{\det(\tau(d_{i,0}),v(d_{i,0}))-\det A_i}{r}=\frac{\det A_i-\det A_i}{r}=0. 
\end{equation}
Combining this with $\eqref{ai}$, we obtain 
\begin{align*}
B(d_{i,0},0)-A_iC_i&=\left(\tau(d_{i,0}), v(d_{i,0})-k'_i\tau(d_{i,0})\right)-(A_iC_{i_{1}},A_iC_{i_{2}})\nonumber\\
&=\left(A_iC_{i_{1}}+\frac{v(d_{i,0})-a_i}{k'_i},a_i-k'_iA_iC_{i_{1}}\right)-(A_iC_{i_{1}},A_iC_{i_{2}})\nonumber\\
&=\left(\frac{v(d_{i,0})-a_i}{k'_i},0\right)\nonumber\\
&=\begin{pmatrix}
    \frac{r-a_{i_1}}{k'_i}&0\\
    \frac{d_{i,0}-a_{i_2}}{k'_i}&0
\end{pmatrix}.
\end{align*}
Consequently,
\begin{equation}\label{e/2-3-part}
    \left \|  B(d_{i,0},0)-A_iC_i\right \|_{\infty}< \varepsilon/2 
\end{equation}
whenever $k'_i\ge \lfloor \frac{2\max\{|r|+|a_{i_1}|,\ 2|r|+|a_{i_2}|\}}{\varepsilon}\rfloor+1$. We now choose $k'_i\in\Z\setminus\{0\}$ satisfying the lower bound above. 
The chosen pair $(k'_i,\ell_i)$ then fixes $C_i$.

Applying 
$\eqref{ai}$, \eqref{B}, \eqref{tau} and \eqref{Bi}, we obtain 
\begin{align}
B(d_i,t_i)-A_iC_i&=\left(\frac{v(d_i)-a_i}{k'_i}+(0,\psi(d_i))^T+t_iv(d_i),-k'_i(0,\psi(d_i))^T-t_ik'_iv(d_i) \right) \nonumber\\
&=\begin{pmatrix}
 \frac{r-a_{i_1}}{k'_i}+t_ir
 &-t_ik'_ir\\
  \frac{d_i-a_{i_2}}{k'_i}+\psi(d_i)+t_id_i&-k'_i\psi(d_i)-t_ik'_id_i
\end{pmatrix}\nonumber
\end{align}
and
\[
\psi(d_i)=\frac{d_i}{r}\left((A_iC_{i_1})_1-\frac{a_{i_1}}{k'_i}\right)-\left((A_iC_{i_1})_2-\frac{a_{i_2}}{k'_i}\right)-\frac{\det A_i}{r}. \]
Then $B(d_i,t_i)-A_iC_i$ depends continuously on $d_i$ and $t_i$. Together with \eqref{e/2-3-part}, this continuity gives
$\eta_{i,0}>0$ and $\theta_i>0$ such that
\[
 \left \|  B(d,t)-A_iC_i\right \|_{\infty}< \varepsilon,\quad
\text{whenever}\
(d,t)\in
(d_{i,0}-\eta_{i,0},d_{i,0}+\eta_{i,0})
\times(-\theta_i,\theta_i).
\]
Set $\eta_i=\min\{\eta_{i,0},|r|\}.$
Then 
\[0<\eta_i\le |r|\lesssim \frac{1}{\sqrt{N}}.\] 
Define
\[
I_i=(d_{i,0}-\eta_i,d_{i,0}+\eta_i),
\quad
J_i=(-\theta_i,\theta_i).
\]
Then
\begin{equation}\label{IJ}
 \left \|  B(d,t)-A_iC_i\right \|_{\infty}< \varepsilon,
 \quad \text{for all } (d,t)\in I_i\times J_i.
\end{equation}
By \eqref{d_{i,0}}, every $d\in I_i$ satisfies $|d|\lesssim \frac{1}{\sqrt{N}}$, and hence \[\left \| v(d) \right \|\le |r|+|d|\lesssim \frac{1}{\sqrt{N}}+\frac{1}{\sqrt{N}}\lesssim \frac{1}{\sqrt{N}},\quad\text{for all}\ d\in I_i.\]

Next, in view of $\eqref{IJ}$,
for each $1\le i\le N$, we recursively choose a suitable pair
$(d_i,t_i)\in I_i\times J_i
$
such that the resulting lattices 
$\Lambda_i=B(d_i,t_i)\Z^2,
$
satisfying
\[
\Lambda_i\cap\Lambda_j=\{0\},
\quad\text{for all }i\ne j.
\]
Choose an arbitrary pair
$(d_1,t_1)\in I_1\times J_1$
and set
$\Lambda_1=B(d_1,t_1)\Z^2.$
Suppose inductively that $\Lambda_1,\ldots,\Lambda_{i-1}$ have already been constructed. We show that one can choose $(d_i,t_i)\in I_i\times J_i$ so that, upon setting
\[
\Lambda_i=B(d_i,t_i)\mathbb Z^2,
\]
we have
\[
\Lambda_i\cap\Lambda_j=\{0\},
\quad\text{for every }j<i.
\]
Suppose, there exist $n,m\in\Z^2\setminus\{0\}$ such that $B(d_i,t_i)n=B(d_j,t_j)m$.
Notice that 
\begin{equation}\label{B(dj,tj)n}
B(d_i,t_i)n=B(d_i,0)n+t_i(n_1-k'_in_2)v(d_i).    
\end{equation}
If $n_1-k'_in_2=0,$ we have $n=n_2(k'_i,1)^T$, and hence 
\[B(d_i,t_i)n=B(d_i,t_i)n_2(k'_i,1)^T=n_2v(d_i).\]
Let 
\[\mathcal{O}_{i,j,n_2}:=\left\{ d_i\in I_i:n_2v(d_i)\in \Lambda_j \right\}\]
and
\[\mathcal{W}_i:=\bigcup_{j<i}\bigcup_{n_2\in \Z\setminus\{0\}}
\mathcal{O}_{i,j,n_2}.\]
Since $\Lambda_j$ is countable, each $\mathcal O_{i,j,n_2}$ is countable.
Thus $\mathcal W_i$ is a null set with respect to the Lebesgue measure.
Hence, we may choose $d_i\in I_i\setminus\mathcal W_i$. Fix $d_i$. For this fixed $d_i$, we have $n_1-k'_in_2\ne0$. Then equation
\eqref{B(dj,tj)n} has at most one solution for $t_i$. Hence the set 
\[\mathcal{T}_{i}:=\bigcup_{j<i}\bigcup_{n,m\in\Z^2\setminus\{0\}}\left\{t_i\in J_i:B(d_i,t_i)n=B(d_j,t_j)m\right\}\]
is countable. Therefore, we may choose $t_i\in J_i\setminus\mathcal{T}_{i}$, thereby completing the construction of $B(d_i,t_i)$. Repeating this construction for $i=2,\ldots,N$, we obtain $N$ lattices $\Lambda_1,\ldots,\Lambda_N$ such that
\[
\Lambda_i\cap\Lambda_j=\{0\},
\quad \text{for all } i\neq j.
\]

\textbf{Step 3}: We construct a function $f\in L^{1}(\R^2)$ that tiles all $\Lambda_i$ and satisfies $\diam \supp f =O(\sqrt{N})$.

Since 
\[\text{vol}(\Lambda_i)=|\det B_i|=|\det A_i|=\omega_{j_i}\in \mathcal S,\]
and $v(d_1),\ldots, v(d_N)$ are primitive vectors of $\Lambda_1,\ldots,\Lambda_N$, respectively, Lemma \ref{disjoint} yields
\begin{equation}\label{big-dual-lattice}
\bigcup_{i=1}^{N}\Lambda^*_i=\bigcup_{i=1}^{N}\left(\left\{\xi\in\Lambda^*_i:\left \langle \xi, v(d_i) \right \rangle\in\Z\setminus\{0\} \right\}\cup \frac{\Z}{\omega_{j_i}}(-d_i,r)^T\right) 
\end{equation}
Choose $u_i\in\R^2$ linearly independent of $v(d_i)$ with $\left \| u_i \right \|\lesssim \frac{1}{\sqrt{N}}$ for all $i$. Define
\[P_i=\left\{zv(d_i)+z'u_i:0\le z< 1, 0\le z'< 1\right\}\]
and 
\[f=\frac{r^2}{\omega_1^2}\mathds1_{[0,\omega_1/|r|)^2}\ast\cdots\ast\frac{r^2}{\omega_s^2}\mathds1_{[0,\omega_s/|r|)^2}\ast \frac{1}{|\det(v(d_1),u_1)|}\mathds1_{P_1}\ast\cdots\ast \frac{1}{|\det(v(d_N),u_N)|}\mathds1_{P_N}.\]
As in the preceding argument, $f$ is the convolution of finitely many functions in $L^1(\R^2)$, and hence $f\in L^1(\R^2)$.
A direct calculation gives that the zero set of the Fourier transform of $\frac{r^2}{\omega_i^2}\mathds1_{[0,\omega_i/|r|)^2}$ is given by
\begin{equation}\label{zero-set-r}
\left\{\left(\frac{r}{\omega_i} p,q\right)^T:p\in \Z\setminus\{0\},q\in \R\right\}\bigcup \left\{\left(p,\frac{r}{\omega_i}q\right)^T:q\in \Z\setminus\{0\},p\in \R\right\}.    
\end{equation}
Similarly, for $1\le i\le N$, the zero set of the Fourier transform of $\frac{1}{|\det(v(d_i),u_i)|}\mathds{1}_{P_i}$ is given by
\begin{equation}\label{zero-set-pi}
\left\{\xi\in\R^2:\left \langle \xi, v(d_i) \right \rangle\in\Z\setminus\{0\}\right\}
\cup
\left\{\xi\in\R^2:\left \langle \xi, u_i \right \rangle\in\Z\setminus\{0\}\right\}.    
\end{equation}
It follows from \eqref{big-dual-lattice}, \eqref{zero-set-r} and \eqref{zero-set-pi} that $\widehat f$ vanishes on
\[\bigcup_{i=1}^N\Lambda_i^*\setminus\{0\}.\] Hence, by \eqref{tile-all-lattices},
$f$ tiles with all lattices $\Lambda_1,\ldots,\Lambda_N$. We now estimate the diameter of $\supp f$. Indeed,
\begin{align*}
\diam\supp f&\le \sum_{i=1}^{s}\frac{\sqrt{2}\omega_i}{|r|}+\sum_{i=1}^{N}\diam P_i\\
&\lesssim \sqrt{N}+\sum_{i=1}^{N}(\left\| v(d_i) \right\|+ \left\| u_i \right\|)\\
&\lesssim \sqrt{N}+\frac{N}{\sqrt{N}}\lesssim \sqrt{N}.    
\end{align*}
This shows that $\diam\supp f=O(\sqrt N)$, thereby completing the proof.
\end{proof}

\bigskip
\noindent\textbf{Acknowledgements}

\noindent
The author is grateful to Professor Mihalis Kolountzakis for his guidance during the author’s visit, for bringing this problem to the author’s attention, and for carefully reading an earlier version of the manuscript.

\printbibliography
\end{document}